\documentclass[a4paper,12pt,oneside,headsepline]{scrartcl}

\newenvironment{myproof}[1][Proof]{\par\noindent\textbf{#1.} }{\hfill$\square$\par}

\usepackage{etoolbox}
\usepackage{ifdraft}
\usepackage{iftex}

\ifluatex
\usepackage[utf8]{luainputenc}
\else
\usepackage[utf8]{inputenc}
\fi
\usepackage[T1]{fontenc}

\usepackage{lmodern}
\usepackage{fourier}

\usepackage{amsfonts}
\usepackage{mathrsfs} 
\usepackage{dsfont}

\usepackage{amssymb}

\usepackage{amsmath}
\usepackage{mathtools}
\usepackage[thmmarks, amsmath, amsthm]{ntheorem}
\usepackage{nccmath}    

\usepackage[draft=false]{scrlayer-scrpage}
\usepackage{needspace}

\usepackage[final,
            pdfborder={0 0 0}, colorlinks=true,
            linkcolor=BrickRed, citecolor=ForestGreen, urlcolor=RoyalBlue]{hyperref}

\usepackage[cmyk,dvipsnames]{xcolor}

\ifdraft{%
\usepackage[textsize=scriptsize,colorinlistoftodos]{todonotes}
\usepackage{lineno}
\usepackage{scrtime}
}{}

\usepackage{booktabs}
\usepackage[inline]{enumitem}
\usepackage{lettrine}
\usepackage{tikz-cd}
\usepackage{url}

\makeatletter
\newcommand{\writelabel}[1]{#1\def\@currentlabel{#1}}
\makeatother

\newcommand{\minwidthmathbox}[2]{%
  \mathmakebox[{\ifdim#1<\width\width\else#1\fi}]{#2}%
}

\newcommand{\tbf}{\bfseries}

\newcommand{\bbC}{\ensuremath{\mathbb{C}}}

\newcommand{\bbP}{\ensuremath{\mathbb{P}}}

\newcommand{\bbR}{\ensuremath{\mathbb{R}}}

\newcommand{\bbZ}{\ensuremath{\mathbb{Z}}}

\newcommand{\mrelspace}[1]{\mathrel{\mspace{#1}}}

\let\rightarroworig\rightarrow
\renewcommand{\rightarrow}
  {\protect\relbar\mrelspace{-9.7mu}\rightarroworig}

\let\leftarroworig\leftarrow
\renewcommand{\leftarrow}
  {\protect\leftarroworig\mrelspace{-9.7mu}\relbar}

\renewcommand{\longrightarrow}
  {\protect\relbar\mrelspace{-3.2mu}\relbar\mrelspace{-9.5mu}\rightarroworig}

\makeatletter
\@ifundefined{xlongrightarrow}{%

}{%

}
\makeatother

\makeatletter
\@ifundefined{xlongleftarrow}{%

}{%

}
\makeatother

\newcommand{\Hom}{\operatorname{Hom}}

\newcommand{\GL}[1]{\ensuremath{\mathrm{GL}_{#1}}}

\newcommand{\SL}[1]{\ensuremath{\mathrm{SL}_{#1}}}

\renewcommand{\det}{\ensuremath{\mathrm{det}}}

\DeclareMathOperator{\PSL}{PSL}

\numberwithin{equation}{section} 
\numberwithin{equation}{section} 
\newtheorem{theorem}{Theorem}[section]
\newtheorem{lemma}[theorem]{Lemma}
\newtheorem{corollary}[theorem]{Corollary}
\newtheorem{example}[theorem]{Example}
\newtheorem{proposition}[theorem]{Proposition}
\newtheorem{remark}[theorem]{Remark}
\newtheorem{definition}[theorem]{Definition}
\newtheorem{claim}[theorem]{Claim}

\title{%
  Deformations of differential equations
}
\author{%
ZIYU ZHANG%
}
\date{}
\begin{document}

 \thispagestyle{scrplain}
\begingroup
\deffootnote[1em]{1.5em}{1em}{\thefootnotemark}
\maketitle
\endgroup


\begin{abstract}
\small
\noindent
{\tbf Abstract:}
We study perturbations of linear differential equations equations, deriving explicit series solutions, using Dyson-type expansions. We analyze the monodromy of deformationed solutions in a number of examples, and relate this to cocycles in a cohomological framework We also analyze the spectral properties of the hypergeometric equation under infinitesimal deformations, derive the first-order eigenvalue correction via orthogonality, and establish natural bounds for the perturbed problem. 

\end{abstract}


\vspace{-1.5em}
\renewcommand{\contentsname}{}
\setcounter{tocdepth}{2}
\tableofcontents
\vspace{1.5em}


\Needspace*{4em}

\addcontentsline{toc}{section}{Introduction}
\section*{Introduction}
The study of linear differential equations, especially those with regular singularities such as hypergeometric equations, is fundamental in both pure and applied mathematics. These equations not only model a wide range of phenomena in mathematical physics but also provide rich structures for investigating monodromy representations, deformation theory, and spectral geometry\cite{beukers1989monodromy}\cite{tveiten2009deformation}\cite{molag2017monodromygeneralizedhypergeometricequation}.\\
In recent decades, the deformation of linear differential equations has attracted significant attention, particularly in understanding how solutions and monodromy groups evolve under perturbations. While classical treatments often assume holomorphic perturbations, or perturb the locations of singularities \cite{mitschi2012monodromy,borodin2004isomonodromy}, this work extends the scope to include more general --- even non-holomorphic --- deformations, thereby revealing deeper geometric and cohomological structures\cite{tveiten2009deformation, 2022}.\\

This paper introduces a novel perspective for interpreting perturbative corrections of linear differential equations. We construct series solutions to perturbed differential equations equations and analyze how their monodromy groups change. The monodromy deformation is interpreted via Dyson-type series expansions, which allow us to express the perturbed fundamental matrix in terms of integral corrections. These corrections are shown to represent 1-cocycles\cite{goldman1988deformation}, giving rise to elements in the first cohomology group, and under suitable conditions, even tangent vectors in Teichmüller space\cite{sernesi2000overview}.\\
Finally, we study the hypergeometric equation under an \emph{infinitesimal deformation}. Exploiting the orthogonality of hypergeometric functions, we derive the first-order correction to the eigenvalue. Because our analysis is confined to the infinitesimal regime, we establish explicit upper and lower bounds for the perturbed problem, which in turn yield a natural admissibility (constraint) condition for the correction. Within this framework, we further examine the Dirac delta $\delta$ as a model perturbation, justify a distributional approximation to its effect, and delineate the range of validity and limitations of the resulting asymptotics.

The main contributions of this paper are as follows:
\begin{itemize}
    \item We derive the general series solution to a deformed differential equation using Dyson-type expansions, and compute the first-order monodromy corrections explicitly in several examples\cite{Argeri:2014qva}.
    \item We prove that first-order monodromy deformations correspond to elements of cohomology.
    \item We analyse the first-order correction of the eigenvalue of the hypergeometric equation, and determine the constraints on the variation of eigenvalues.
\end{itemize}

\section*{Acknowledgment}
I would like to express my deepest gratitude to Dr. Adam Keilthy for his invaluable guidance and support throughout the course of this research. His insights into the deformation theory of hypergeometric equations—especially regarding series solutions, matrix representations, and monodromy transformations—have played a crucial role in shaping the core structure of this work.

I am sincerely thankful for his mentorship, which has been both rigorous and generous, and which has profoundly impacted my academic journey.

\section{Preliminaries}
\subsection{Generalised Hypergeometric Equations} \ \ The generalized hypergeometric differential equation is an $n$-th order linear differential equation with regular singularities at $x = 0$, $x = 1$, and $x = \infty$. It is commonly written in the form
\begin{equation}
\left[ \prod_{j=1}^{p} \left( x \frac{d}{dx} + a_j \right) -  \frac{d}{dz}\prod_{k=1}^{q} \left( x \frac{d}{dx} + b_k -1 \right) \right] y(x) = 0
\label{eq 1.1}
\end{equation}
where $\{a_j\}_{j=1}$ and $\{b_k\}_{k=1}$ are complex parameters and very often we choose q=p-1.
In some cases, the equation${\eqref{eq 1.1}}$ can be written into general form
\begin{equation}
  \sum^p_{n=0}\sum^p_{k=n}c(k,n)e_{p-k}(a_1,..)z^n\frac{d^ny}{dz^n} -\sum_{n=1}^{p} 
\sum_{k=n-1}^{p-1}
e_{p-k-1}(b_1-1,...)
\sum_{j=n-1}^{k}
\binom{k}{j}
c(j, n-1)
z^{n-1} \frac{d^n y}{dz^n}=0
  \label{eq1.2}
\end{equation}
where c(x,k) is the Stirling number of the secondkind \cite{mansour2016commutation} and ${e_{p-k-1}
  (a_1,...,a_n)}$ are the elementary symmetry polynomials \cite{macdonald1998symmetric}.

The derivation of the equation ${\eqref{eq1.2}}$ is shown in Appendix A, though we will only need the case of $(p,q)=(2,1)$.
A solution is given by the generalized hypergeometric series\cite{beukers1989monodromy}:
\begin{equation}
{}_pF_{q} \left( \begin{array}{c}
a_1, \dots, a_p \\
b_1, \dots, b_{q}
\end{array}; x \right)
= \sum_{m=0}^{\infty} \frac{(a_1)_m \cdots (a_p)_m}{(b_1)_m \cdots (b_{q})_m} \cdot \frac{x^m}{m!},
\end{equation}
where $(a)_m$ denotes the Pochhammer symbol:
\begin{equation}
(a)_m = a(a+1)\cdots(a+m-1) = \frac{\Gamma(a + m)}{\Gamma(a)}.
\end{equation}

The hypergeometric differential equation has singularities only at $x = 0$, $x = 1$, and $x = \infty$. These are all regular singular points, and so the local behavior of solutions near them is governed by Frobenius theory.
\begin{example}
    From equation ${\eqref{eq1.2}}$, the second order hypergeometric equation can be written as
    \begin{equation}
       x(1 - x) y'' + [c - (a + b + 1)x] y' - ab y =0 
       \label{eq1.3} 
    \end{equation}
    Assuming $a,b,c$ are sufficiently generic \cite{riley2006mathematical}, the general solution of ${\eqref{eq1.3}}$ around 0 is 
    \begin{equation}
      y(x)= A{}_2F_1(a, b; c; x) +  Bx^{1 - c} {}_2F_1(a - c + 1, b - c + 1; 2 - c; x). 
      \label{1.7}
    \end{equation}
    Similarly, the general solution around 1 is 
    \begin{equation}
       y(x) =A{}_2F_1(a, b; a + b - c + 1; 1 - x) +B(1 - x)^{c - a - b} {}_2F_1(c - a, c - b; c - a - b + 1; 1 - x), \quad \text{if } c - a - b \notin \mathbb{Z} 
       \label{equation 1.8}
    \end{equation}
\end{example}
This will be a recurring example throughout this article.\\  
\subsection{Dyson Type Series}
In quantum field theory, using the Dyson series to expand the time evolution is a powerful tool for solving differential equations arising from physical problems. Here, we introduce the Dyson-type expansion of the integral to address some computational problems in this research.\\
\begin{definition}
    Any linear differential equation
    \[y^{(n)} + A_{n-1}y^{(n-1)} + \cdots + A_1y^\prime + A_0y = 0\]
    can be written as
    \begin{equation}
        \begin{pmatrix}
            y^{\prime}\\
            y''\\
            \vdots\\
            y^{n}
    
        \end{pmatrix} =\begin{pmatrix}

           0 & 1 & 0&\cdots &0\\
            0 & 0 &  1& \cdots & 0\\
          
            \vdots & \vdots & \ddots & \vdots\\
            0 & \cdots & 0  & 0 & 1\\
             -A_{n-1} & -A_{n-2} & & \cdots & -A_0\\
        \end{pmatrix} \begin{pmatrix}
                      y\\
            y^{\prime}\\
            \vdots\\
            y^{n-1}
     \end{pmatrix}
     \label{eq1.4}
    \end{equation}
    \label{def ma 1}
\end{definition}
Let \( \{y_1(x), \dots, y_n(x)\} \) be a fundamental system of solutions to the corresponding homogeneous equation \( L[y] = 0 \).\\
In what follows, we often write
\[\vec{\psi}=\begin{pmatrix}
                   y\\
            y^{\prime}\\
            \vdots\\
            y^{n-1}
     \end{pmatrix}\]
And so we can write our differential equation as
\[\frac{d\vec{\psi}}{dx} = A(x)\vec{\psi}.\]
From \cite{Argeri:2014qva}, a (possibly divergent) series  solution of ${\Vec{\psi}}$ is given by
   \begin{align}\label{Dyson 1}
       \Vec{\psi} &=\mathcal{P} exp(\int_{x_0}^{x}dxA(x))\Vec{\psi_0}\\
       &= \left(I+\left(\int^{x}_0 A(x_1)dx_1\right) + \left(\int^x_0A(x_2)dx_2\int^x_0 A(x_1)dx_1\right)+\cdots\right)\Vec{\psi_0}
   \end{align}
for some initial vector $\vec{\psi}_0$. 

\begin{theorem}
   The ordinary differential matrix equation
    \begin{equation}
        \frac{d \Vec{\psi}}{dx} = \rho A\Vec{\psi}
        \label{Dyson 2}
    \end{equation}
admits a formal series solution
\begin{equation}
   \Vec{\psi} = \Vec{\psi_0} + \left(\left(\int^{x}_0 A(x_1)dx_1\right)\rho + \left(\int^x_0A(x_2)dx_2\int^x_0 A(x_1)dx_1\right)\rho^2+\cdots\right)\Vec{\psi_0}
   \label{Dyson 3}
\end{equation}
\end{theorem}
\begin{remark}
    Here we treat $\rho$ as a formal variable, or as a sufficiently small complex parameter, to ensure convergence of the series in some sense.
\end{remark}

\subsection{The fundamental matrix and gauge transformations} 
An ordinary differential equation of order $n$
\[y^{(n)} + A_{n-1}(x)y^{(n-1)} + \cdots + A_1(x)y^{(1)} + A_0(x)y = 0\]
has a basis of $n$ linearly independent solutions, under very mild conditions. 

\begin{definition}
Given a basis $y_1,y_2,\ldots,y_n$ of solutions, we define a fundamental matrix of solutions
\begin{equation}
W(x) = \begin{pmatrix}
y_1(x) & y_2(x) & \cdots & y_n(x) \\
y_1'(x) & y_2'(x) & \cdots & y_n'(x) \\
\vdots & \vdots & \ddots & \vdots \\
y_1^{(n-1)}(x) & y_2^{(n-1)}(x) & \cdots & y_n^{(n-1)}(x)
\end{pmatrix}
\end{equation}
satisfying
\[W^\prime(x)= A(x)W(x)\]
where $A(x)$ is as before.
\end{definition}

We will often talk about ``the'' fundamental matrix of a differential equation, even though $W(x)$ is only defined up to a (constant) change of basis. Non-constant changes of basis give solutions to a distinct, but equivalent differential equation, often referred to as a gauge transformation.

\begin{definition}
    Two differential equations
    \[\vec{\psi}^\prime = A_1(x)\vec{\psi},\quad \vec{\phi}^\prime = A_2(x)\vec{\phi}\]
    are called gauge equivalent if there exists an invertible matrix $Z(x)$ such that 
        \[A_1(x) = Z^\prime(x)Z^{-1}(x) + Z(x)A_2(x)Z^{-1}(x)\]
\end{definition}

Given two gauge equivalent differential equations, $\vec{\phi}$ is a solution to $\vec{\phi}^\prime = A_2(x)\vec{\phi}$ if and only if $\vec{\psi}=Z(x)\vec{\phi}$ is a solution to $\vec{\psi}^\prime = A_1(x)\vec{\psi}$. A particularly useful type of gauge transformation is given by the inverse of a fundamental matrix.

\begin{lemma}\label{lem:gauge-transformation}
    Assume that the differential equation
    \[\vec{\psi}^\prime = A(x)\vec{\psi}\]
    has a fundamental matrix of solutions $W(x)$ in an open set $U$. Then
    \[\vec{\psi}^\prime = A(x)\vec{\psi} + B(x)\vec{\psi},\quad\text{and}\quad \vec{\psi}^\prime = W^{-1}(x)B(x)W(x)\vec{\psi}\]
    are gauge equivalent in $U$.
\end{lemma}
\begin{myproof}
    Let $\vec{\psi}$ be a solution to
    \[\vec{\psi}^\prime = A(x)\vec{\psi}+B(x)\vec{\psi}\]
    in $U$. Then consider $\vec{\phi}:=W^{-1}(x)\vec{\psi}$. As
    \[ 0 = \left(W(x)W^{-1}(x)\right)^\prime = W^\prime(x)W^{-1}(x) + W(x)(W^{-1}(x))^\prime\]
    we have that
    \begin{align*}
        \vec{\phi}^\prime &= W^{-1}(x)\vec{\psi}^\prime - W^{-1}(x)W^\prime(x)W^{-1}(x)\vec{\psi}\\
        &= W^{-1}(x)A(x)\vec{\psi} +W^{-1}(x)B(x)\vec{\psi} - W^{-1}(x)A(x)W(x)W^{-1}(x)\vec{\psi}\\
        &= W^{-1}(x)B(x)\vec{\psi}=W^{-1}(x)B(x)W(x)\vec{\phi}
    \end{align*}
\end{myproof}

\subsection{The monodromy group} 



In general, a linear differential equation with singularities does not have globally defined solutions. Instead, we obtain local solutions which can be glued together to obtain multivalued solutions with non-trivial monodromy around the singularities of the differential equation. As such, to every differential equation induces a representation of a fundamental group, called the monodromy representation.

\begin{definition}
Consider an $n\times n$ linear differential system
\begin{equation}
\frac{dY}{dx} = A(x) Y(x),
\end{equation}
where $A(x)$ is a meromorphic matrix-valued function with poles at $\{a_1, \dots, a_k\}$. Fixing a basis $\{Y_1,\ldots,Y_n\}$ of solutions, analytic continuation along closed loops in $X=\bbP^1\setminus\{a_1,\ldots,a_k\}$ defines a group homomorphism
\begin{equation}
\rho: \pi_1(X, x_0) \longrightarrow \mathrm{GL}_n(\mathbb{C}),
\end{equation}
called the \emph{monodromy representation}\cite{beukers1989monodromy} of the system. Its image
\[
\mathrm{Mon}(A) := \mathrm{Im}(\rho)
\]
is called the \emph{monodromy group}
\end{definition}

In the context of Fuchsian systems, all singularities are regular, and exponent matrices control the local behavior of solutions near each singular point. The monodromy matrices then reflect how these local solutions are patched together to form global ones\cite{tveiten2009deformation}. Understanding their variation under perturbation is a central theme of this work.

 \begin{remark}
 Given a monodromy representation $\rho: \pi_1(X) \to \mathrm{GL}_n(\mathbb{C})$, this is equivalent to the data of a local system $V$ on $X$. Then the first cohomology group  $H^1(X, \mathrm{End}(\mathcal{V}))$\cite{tveiten2009deformation} describes infinitesimal deformations of $\rho$ up to equivalence. It also appears naturally in deformation theory, gauge theory, and the study of flat connections.
 \end{remark}

\subsection{Teichmüller Space}
Teichmüller space\cite{2022} $\mathcal{T}(S)$ parametrises the space of complex structures on Riemann surfaces, and plays a central role in the deformation theory of Riemann surfaces. It is closely related to deformations of monodromy representations, and so provides the natural setting to interpret how analytic and non-holomorphic perturbations of differential equations induce variations in monodromy data.

\begin{definition}
Let $S$ be a topological surface of genus $g \geq 2$. A marked Riemann surface is a pair $(X,f)$ of a Riemann surface and 
\[f:S\to X\]
is an orientation-preserving homeomorphism. The \emph{Teichmüller space} $\mathcal{T}(S)$ is the space of equivalence classes of marked Riemann surfaces, defined as
\[
\mathcal{T}(S) := \left\{ (X, f)\right\}/ \sim
\]
where $(X_1, f_1) \sim (X_2, f_2)$ if there exists a biholomorphism $\phi: X_1 \to X_2$ such that $\phi \circ f_1$ is homotopic to $f_2$.
\end{definition}

If $S$ is compact, Teichm{\"u}ller space is in bijection with the set of injective representations
\[\pi_1(S)\to \PSL_2(\bbR)\]
with discrete image, up to conjugation.

This leads naturally to a notion of higher Teichm{\"u}ller theory \cite{thomas2020highercomplexstructureshigher}
in which we study the character variety
\[\Hom(\pi_1(S),G)/\sim\]
of maps to a semi-simple Lie group, up to conjugation. This is equivalent to the moduli space of flat $G$-connections on $S$ \cite{Ho2004}.

For families of linear differential equations over Riemann surfaces, the monodromy representation gives a homomorphism
\[
\pi_1(S \setminus \{p_i\}) \to \mathrm{GL}_n(\mathbb{C}),
\]
and its equivalence class (modulo conjugation) defines a point in the character variety.

\section{The general solution and monodromy group structure of deformed differential equations }

We define a deformation of the differential equation
\[y^{(n)} + A_{n-1}y^{(n-1)} + \cdots + A_1y^\prime + A_0y = 0\]
to be a differential equation of the form
\[y^{(n)} + A_{n-1}y^{(n-1)} + \cdots + A_1y^\prime + A_0y + \rho f(x)y = 0\]
for some sufficiently nice $f(x)$.
We can rewrite this as
    \begin{equation}
        \begin{pmatrix}
             
            y^{\prime}\\
            y''\\
            \vdots\\
            y^{n}
    
        \end{pmatrix} =\left(\begin{pmatrix}
         0 & 1 & 0&\cdots &0\\
            0 & 0 &  1& \cdots & 0\\
          
            \vdots & \vdots & \ddots & \vdots\\
            0 & \cdots & 0  & 0 & 1\\
             -A_{n-1} & -A_{n-2} & & \cdots & -A_0\\
        \end{pmatrix}+\begin{pmatrix}
         0& 0 & \cdots & 0\\
            0 & 0 & \cdots & 0\\
            0 & 0 &  \cdots &0\\
            \vdots & \vdots & \ddots & \vdots\\
            0 & \cdots & 0 & -\rho f(x)
        \end{pmatrix}\right) \begin{pmatrix}
                   y\\
            y^{\prime}\\
            \vdots\\
            y^{n-1}
     \end{pmatrix} 
     \label{eq 2.27}
    \end{equation}
    
In this section, the goal is to find a formal series solution of deformed differential equations and hence analyze the monodromy group of these deformations in special cases. 

We will derive the series solution using Dyson series, but a recursive approach using variation of parameters can also be used to explicitly determine a series solution, as explained in Appendix B.

More generally, we can consider deformations of the form
\begin{equation}
    \frac{d}{dx}\ \Vec{\psi} = A \Vec{\psi} + \rho B \Vec{\psi}
    \label{mat 1}
\end{equation}
A is the original coefficient matrix, and B is called the perturbation matrix with parameter ${\rho}$. 

 \subsection{Series solution to deformed differential equations}
In all that follows, we assume that $\frac{d}{dx}\ \Vec{\psi} = A \Vec{\psi}$ admits $n$ linearly independent solutions, so that the fundamental matrix is generically invertible.

\begin{theorem}
Under gauge transformation given by $Z=W^{-1}$, the deformed differential equation
\begin{equation}
    \frac{d}{dx}\ \Vec{\psi} = A \Vec{\psi} + \rho B \Vec{\psi}
\end{equation}
is equivalent to
\begin{equation}
    (Z\Vec{\psi})' = \rho ZBZ^{-1}(Z\Vec{\psi})
    \label{2.24}
    \end{equation}
\end{theorem}
\begin{myproof}
For any $Z$, we have that
    \begin{equation}
        (Z\Vec{\psi})' = Z'\Vec{\psi} +Z \Vec{\psi}'
        \label{2.26}
    \end{equation}
The by using equation ${\eqref{mat 1}}$, the equation ${\eqref{2.26}}$ can be written as
\begin{equation}
    (Z\Vec{\psi})' = Z'\Vec{\psi} + ZA \Vec{\psi} + \rho ZB \Vec{\psi}
    \label{2.27}
\end{equation}
Assuming $Z$ is invertible, equation ${\eqref{2.27}}$ can be written as
\begin{equation}
     (Z\Vec{\psi})' = Z'Z^{-1}Z\Vec{\psi} + ZA Z(Z^{-1}\Vec{\psi}) + \rho ZBZ^{-1} (Z\Vec{\psi})   
     \label{2.28}
\end{equation}
Next, note that
\[\frac{dW}{dx}=AW\]
Letting $Z=\tilde{W}^{-1}$, we find
\[Z'W+ZAW=0\]
and hence $Z'+ZA=0$. Thus, if $Z=W^{-1}$, $Z\vec{\psi}$ satisfies
\[(Z\Vec{\psi})' = \rho ZBZ^{-1}(Z\Vec{\psi})\]
\end{myproof}

Using this gauge transformation, we can easily find a formal series solution to the deformed differential equation using Lemma \ref{lem:gauge-transformation}.

\begin{theorem}
    Let $W_\rho$ be a fundamental matrix of
    \[\frac{d}{dx}\ \Vec{\psi} = A \Vec{\psi}+\rho B\vec{\psi}.\]
    Then $W_\rho$ can be written as a Dyson-type series
    \begin{equation}
     W_{\rho} ={W} \left( I + \left( \int Z B Z^{-1} \right) \rho 
+ \left( \int Z B Z^{-1} \int Z B Z^{-1} \right) \rho^2 + \cdots \right)  
\label{Deformed Wronskian}
    \end{equation}
\end{theorem}

\subsection{Deformations of the monodromy}
In the case of a Fuchsian differential equation, we can consider how the monodromy group varies under the perturbation, as least first to order. We will investigate this in several cases, starting with an example.

\subsubsection{Trivial deformations of hypergeometric equations}
Consider a simple deformation of the hypergeometric differential equation
\[x(1-x)y^{\prime\prime} + (c-(a+b+1)x)y^\prime -aby = 0\]
where $a,b,c\in\bbC$ are taken to be generic complex parameters, given by
\begin{equation}
\Vec{\psi}^\prime =   \begin{pmatrix}
0 & 1\\
\frac{-(a + b + 1)x + c}{x(1-x)} & \frac{ab}{x(1-x)} 
    \end{pmatrix}\Vec{\psi} + \rho \begin{pmatrix}
 0 & 0\\
 \frac{1}{x(1-x)} & 0
\end{pmatrix} \Vec{\psi}
\end{equation}

Around $0$, a fundamental matrix of solutions to
\[\vec{\psi}^\prime = \begin{pmatrix}
1 & 0\\
\frac{-(a + b + 1)x + c}{x(1-x)} & \frac{ab}{x(1-x)} 
    \end{pmatrix}\vec{\psi}
    \]
is given by
\[W(x) = \begin{pmatrix} y_1(x) & y_2(x) \\ y_1^\prime(x) & y_2^\prime(x)\end{pmatrix}\]
where
\[y_1(x) = {}_2F_{1} \left( \begin{array}{c}
a, b \\
c
\end{array}; x \right)\]
and
\[y_2(x) = x^{1-c}{}_2F_1\left(\begin{array}{c}a-c+1, b-c+1\\ 2-c\end{array};x\right).\]

     From equation \eqref{Deformed Wronskian}, the perturbed solutions are given to  first-order by
     \[
         {W} + \rho {W}\int^x_0{W}^{-1}B{W}dt
     \]
     Therefore, for some non-zero $A$ 
     \begin{equation}
           W_\rho - {W} = A{W}\int^x_0(1-t)^{a+b-c}\begin{pmatrix}
               y_1y_2&y^2_2\\
               -y^2_1&-y_1y_1
           \end{pmatrix}dt
           \label{2.47}
     \end{equation}
    which is in turn equal to
     \begin{equation}
     {W}\int^x_0t^{c-1}\begin{pmatrix}
         t^{1-c}H_1 & t^{2-2c}H_2\\
         H_3  & t^{1-c}H_4
     \end{pmatrix}dt
     \label{2.48}
     \end{equation}
     for some holomorphic $H_1,\ldots,H_4$. After integrating, we obtain a matrix of the form
\[{W}\begin{pmatrix}
\Tilde{H_1} & x^{2-c}\Tilde{H_2}\\
    x^c\Tilde{H_3} & xH_4
\end{pmatrix}\]
for some holomorphic $\tilde{H}_1,\ldots,\tilde{H}_4$.
Thus, we find
\begin{equation}
\begin{cases}
    y_{1\rho} = y_1 + \rho(y'_1(x\Tilde{H_1}+1)+y'_2x^2\Tilde{H_3})+\mathcal{O}(\rho^2)\\
    y_{2\rho} = y_2+ \rho( y'_1x^{2-c}\Tilde{H_2}+y'_2x(\Tilde{H_4}+1))+\mathcal{O}(\rho^2)
\end{cases}
\label{equation 2.47}
\end{equation}
We see that, to first order in $\rho$, $y_{1\rho}$ is a holomorphic function function of $x$, while $y_{2\rho}$ has non-trivial monodromy around $0$. Explicitly, the monodromy matrix is
\begin{equation}
    \begin{pmatrix}
        1 & 0\\
        0 & e^{-2\pi i c}
    \end{pmatrix}
\end{equation}
which is exactly the monodromy of the (unperturbed) hypergeometric differential equation around $0$.\cite{molag2017monodromygeneralizedhypergeometricequation}.\\  

    Similarly, we can determine the monodromy around 1. 
In this case, the difference in equation \eqref{2.47} can be computed to be 
\[\int^x_0\begin{pmatrix}
         \mathcal{A}_1 & (1-t)^{c-a-b}\mathcal{A}_2\\
         (1-t)^{a+b-c}\mathcal{A}_3  & \mathcal{A}_4
     \end{pmatrix}dt\]
which integrates to
\[\begin{pmatrix}
    \Tilde{\mathcal{A}}_1 & (1-x)^{c-a-b+1}\Tilde{\mathcal{A}}_2\\
    (1-x)^{a+b-c+1}\Tilde{\mathcal{A}}_3 & \Tilde{\mathcal{A}}_4
\end{pmatrix}\]
for some holomorphic functions $\tilde{\mathcal{A}}_i$.As before, we can get the perturbed monodromy group around 1.
\begin{equation}
    M_{1\rho} =\begin{pmatrix}
        1 & 0\\
        0 & e^{2 \pi i(c-a-b)}
    \end{pmatrix}
\end{equation}   
Thus, this perturbation induces a trivial deformation of the monodromy of the hypergeometric differential equation.

\subsubsection{Non-trivial deformations}
In what follows, we fix a base Fuchsian differential equation $\vec{\psi}^\prime= A\vec{\psi}$, and consider how its monodromy varies under different types of perturbation.

For more complicated perturbation matrices $B$, we can obtain more interesting changes in the monodromy matrices. Here we will discuss a number of cases and derive a general formula for the first order deformation of the monodromy matrices.

As the computation is the same around every singularity, we will only compute the monodromy around $0$. For a multivalued function $F$, we will denote a fixed branch by $F(x)$, and denote by $F(xe^{2\pi i})$ the branch obtained by looping around $0$ once.

\begin{theorem}
For a generic perturbation matrix $B$, the monodromy matrix around $0$ is given by
\begin{equation}
       M_{0,\rho}= M_0+ \left(M_0 C(xe^{2 \pi i}) M_0^{-1} - C(x)\right)M_0\rho + \mathcal{O}(\rho^2)
    \label{equation2.53}
\end{equation}
\label{theorem 2.19}
where
\[
    C(x) = \int^x_0 {W}^{-1}B{W}dt
\]
and $M_0$ is the monodromy around $0$ of $\vec{\psi}^\prime= A\vec{\psi}$.
\end{theorem}

\begin{myproof}
On one hand, computing the value of $W_\rho$ along a loop around $0$, we find 
\begin{equation}
    W_{\rho}(xe^{2 \pi i}) = {W}(xe^{2 \pi i})(I + \rho C(xe^{2 \pi i})) + \mathcal{O}(\rho^2) {W}M_0(I + \rho C(xe^{2 \pi i}))+\mathcal{O}(\rho^2)
    \label{equation 2.54}
\end{equation}
On the other hand, the deformation of the monodromy is defined by
 \[ {W}_\rho M_{0,\rho} =W_{\rho}(xe^{2 \pi i})\]
Comparing these, we find that, to first order,
\[{W}(I+\rho C(x))M_{0,\rho}= {W}M_0(I + \rho C(xe^{2 \pi i}))\]
and so 
\[M_{0,\rho}=(I+\rho C(x)+\mathcal{O}(\rho^2)))^{-1}M_0(I + \rho C(xe^{2 \pi i})+\mathcal{O}(\rho^2))\]
As 
\[(I+\rho C(x)+\mathcal{O}(\rho^2)))^{-1} = I-\rho C(x) +\mathcal{O}(\rho^2)\]
the result follows
 \end{myproof}
    \begin{remark}
        It will be convenient to introduce the notation 
        \[\Delta_a C(x) = M_a C(xe^{2\pi i})M_a^{-1} - C(x).\]
    \end{remark}

In all the following examples, we consider deformations of a differential equation
\[\vec{\psi}^\prime = A(x)\vec{\psi}\]
with fundamental matrix $W(x)$, and monodromy matrix $M_0$ around $0$. We also fix a matrix $H(x)$ with holomorphic entries.

\begin{example}[Branch-Cut deformation]
    Consider the perturbed system:
\[
\Vec{\psi}' = \left( A(x) + \rho\, x^\lambda H(x) \right) \ \vec{\psi},
\]

The correction to monodromy is determined by
\[
C(x) = \int^x_0 Z(t) B(t) Z^{-1}(t)\, dt = \int^x_0 Z(t)\, t^\lambda H(t)\, Z^{-1}(t)\, dx.
\]
Choosing appropriate branches of $W(x)$ and $x^\lambda H(x)$, we can compute $C(xe^{2\pi i})$:
\[C(xe^{2\pi i})= \int_0^x M_0^{-1}W^{-1}(t)B(t-)W(t)M_0 \, dt\]
and hence.
\[M_0C(xe^{2\pi i})M_0^{-1} = e^{2\pi i \lambda}\int_0^x t^\lambda W^{-1}(t)H(t)W(t)\, dt\]
Thus
\[\Delta_0C(x) = (e^{2\pi i \lambda}-1)C(x)\]
\end{example}

\begin{example}[Logarithmic deformation]
    Consider the perturbation:
\[
B(x) = \log(x) \cdot H(x)
\]
Then the Dyson correction term is:
\[
C(x) = \int^x_0 Z(t) B(t) Z^{-1}(t) dx = \int^x_0 Z(t) \log(t) H(t)Z^{-1} dt
\]

Choosing appropriate branches of $W(x)$ and $\log(t) H(x)$, we find
\[M_0C_a(x)M_0^{-1} = \int_0^x (\log(t+2\pi i)W^{-1}(t)H(t)W(t)\, dt\]
and so 
\[\Delta_0 C(x) = 2\pi i C(x).\]

\end{example}

These types of deformation do not preserve Fuchsian-ness of the differential equation, which is reflected in the $x$-dependency of $\Delta_0C(x)$. In contrast, for a meromorphic deformation, we obtain genuine constant monodromy matrices.

\begin{example}[Meromorphic deformations]

Consider the perturbed system:
\[
\Vec{\psi}^\prime = \left( A(x) + \rho\, \frac{1}{x^n} H(x) \right) \ \vec{\psi}.
\]

By Theorem \ref{theorem 2.19}, the monodromy around $0$ is given to first order by
\[
M_{0,\rho} =  \left( I + \rho \Delta_0C(x) \right)M_0\].

As $B(x)=\frac{1}{x^n}H(x)$ is meromorphic, by choosing appropriate branch cuts of $W(x)$, we find that
\begin{align*}
\left(M_0 C_a(x)M_0^{-1} - C(x)\right)^\prime & = M_0M_0^{-1}W^{-1}(x)B(x)W(x)M_0M_0^{-1}-W^{-1}(x)B(x)W(x)=0
\end{align*}
and so $\Delta_0C(x)$ is constant. In particular, if $C(0)$ is well defined, we must have $\Delta_0C(x)=0$, as in our hypergeometric example. Similarly, if $W^{-1}(x)B(x)W(x)$ is meromorphic, then $\Delta_0C(x)$ can be computed via residues.
\end{example}

\begin{remark}\label{rem:log-example}
    In general $\Delta_0C(x)\neq 0$. For example, consider the differential equation
    \[\vec{\psi}^\prime= \begin{pmatrix} 0 & 1\\ 0 & -z^{-1}\end{pmatrix}\vec{\psi}.\]
    One can easily compute that
    \[W(x)=\begin{pmatrix} 1 & \frac{1}{2\pi i}\log(z)\\ 0 & \frac{1}{2\pi i z}\end{pmatrix},\quad B(x)=\frac{1}{x}\begin{pmatrix}1 & 1\\ 0 & 1\end{pmatrix},\quad C(x)=\begin{pmatrix}\log(z) & \frac{-1}{2\pi iz}\\ 0 & \log(z)\end{pmatrix}\]
    and hence
    \[\Delta_0C(x) = \begin{pmatrix} 2\pi i & 0 \\ 0 & 2\pi i\end{pmatrix}\]
\end{remark}

\subsubsection{The cohomology theory of monodromy deformations}
From these deformations of monodromy, we easily obtain a representation of a classifying cocycle of the deformation, viewed as a deformation of a connection \cite{tveiten2009deformation}
\begin{theorem}[Cocycle from Gauge Jump]
Let \( S \) be a Riemann surface, and let
\[
\nabla^{(0)} = d + A(x)
\]
be a flat connection with monodromy representation \( \rho^{(0)} : \pi_1(S) \to \mathrm{GL}_n(\mathbb{C}) \). Consider a deformation
\[
\nabla^{(\rho)} = d + A(x) + \rho B(x)
\]
with meromorphic matrix-valued function \( B(x)\). Define
\[
C(x) := \int^x_0 {W}^{-1}(t) B(t) {W(t)} \, dt.
\]

Then the jump
\[
\Delta_a C := M_aC((x-a) e^{2\pi i}+a)M^{-1}_a - C(x)
\]
defines a 1-cocycle valued in the adjoint representation, and represents a class in the first cocohmology group:
\[
[\Delta_a C] \in H^1(S, \mathrm{ad}(\rho^{(0)})).
\]
\end{theorem}

\begin{myproof}
Following \cite{goldman1988deformation}, for every deformation
\[M_\rho:\Gamma\to \GL[n](\bbC[[\rho]])\]
of the monodromy representation $M:\Gamma\to \GL[n](\bbC)$, there exists a map $X:\Gamma\to \rho\mathfrak{gl}_n[[\rho]]$ such that
\[M_{a,\rho} = \exp(X(\gamma_a))M_a\]
where we denote by $\gamma_a$ the loop around $a$. Furthermore, writing
\[X(\gamma) = \delta(\gamma) \rho + \cdots\]
the map $\delta$ satisfies
\[\delta(\gamma) - \delta(\gamma\eta) + \mathrm{ad}_{M}(\gamma)\delta(\eta)=0\]
giving an element of cohomology 
\[\delta\in H^1(\bbP^1\setminus\{0,1,\infty\},\mathrm{ad}_M).\]
Theorem \ref{theorem 2.19} tells us that the function
\[\delta(\gamma_a) = \Delta_a C\]
gives such a cocycle. Recall that this is independent of $x$ if $B(x)$ is meromorphic. Indeed, we have that
\begin{align*}
    &{}\delta(\gamma_a) - \delta(\gamma_a\gamma_b) + \mathrm{ad}_M(\gamma_a)\delta(\gamma_b)\\
    =&{} M_aC(\gamma_a\cdot x)M_a^{-1} - C(x) - M_aM_bC(\gamma_a\gamma_b\cdot x)M_{b}^{-1}M_{a}^{-1} +C(x)\\
    &+ M_aM_b(C(\gamma_a\gamma_b\cdot x)M_b^{-1}M_a^{-1} - M_aC(\gamma_a\cdot x)M_a^{-1} = 0
\end{align*}
\end{myproof}

The space of such cocycles can be identified with the tangent space of the character variety $\Hom(\Gamma,\GL{n}(\bbC))/\sim$ \cite{goldman1988deformation}. In particular, if a deformation of a differential equation has a monodromy cocycle $\Delta_\bullet C:\Gamma\to \SL{2}(\bbZ)$, then we can identify the corresponding cocycle with a tangent vector to Teichm{\"u}ller space and hence a deformation of the complex structure.

\begin{corollary}[Teichmüller Tangent Vector from Monodromy Deformation]
Let \( \nabla^{(\rho)} = d + A(x) + \rho B(x) \) be a meromorphic deformation of a flat connection over a punctured Riemann surface \( S \), and define the gauge correction term:
\[
C(x) := \int^x_0 {W}^{-1}(t) B(t) {W}(t) \, dt, \quad \Delta_a C := C((x-a)e^{2\pi i}+a) - C(x),
\]
where \( {W}(x)^{-1} \) is the inverse fundamental matrix solution of the unperturbed system.

If $\Delta_a C\in \SL[2](\mathbb{R})$, then it can be identified with a tangent vector in Teichm{\"u}ller space, and hence induces a deformation of the complex structure of $S$.
\end{corollary}

\begin{remark}
    This cocycle encodes an infinitesimal deformation of the monodromy representation, but we construct a full formal deformation of the fundamental matrix and one could ask about the formal deformation theory. As $S=\bbP^1\setminus\{a_1,\ldots,a_n\}$ is a Stein manifold, it will have vanishing second cohomology $H^2(S,\mathrm{ad}_M)=0$, and so every infinitesimal deformation can be extended to a formal deformation. However, computing these extensions is beyond the scope of this article.
\end{remark}

\section{Variation of hypergeometric eigenvalues}

In this section, we consider the second-order hypergeometric equation \eqref{eq1.3}. However, we rewrite it by using the differential operator
\begin{equation}
      x(1 - x) y'' + [c - (a + b + 1)x] y' - ab y =(\widetilde{L}-ab)y=0
      \label{3.1}
\end{equation}
where
    \begin{equation}
        \widetilde{L}= x(1-x)\frac{d^2}{dx^2}+[c - (a + b + 1)x]\frac{d}{dx}
    \end{equation}

A solution to the hypergeometric differential equation then gives an eigenfunction of $\widetilde{L}$
\begin{equation}
    \widetilde{L}y = aby 
    \label{3.3}
\end{equation}
with eigenvalue $ab$.

If we deform $\widetilde{L}$ by a continuous function, we can consider how this eigenvalue varies:

\begin{equation}
[\widetilde{L}+\rho f(x)]\Tilde{y}=\Tilde{\lambda}_\rho\Tilde{y}]
\label{3.4}
\end{equation}
Expanding both $\tilde{\lambda}_\rho$ and $\tilde{y}$ as formal series in $\rho$, we can use Theorem \eqref{2.15} to recursively solve for both $\tilde{\lambda}_\rho$ and $\tilde{y}$.

Inspired by perturbation theory in quantum mechanics \cite{griffiths2018introduction}, in this section we explore the corrections to the eigenvalues, which we believe to be original to the author.

We first note an orthogonality property for solutions to the hypergeometric equation \cite{riley2006mathematical}.

\begin{lemma}\label{ortho}
The solutions
   \[y_1(x)= {}_2F_1\left(\begin{array}{c}a,b\\ c\end{array};x\right), \quad \text{and}\quad y_2(x)=x^{1-c}{}_2F_1\left(\begin{array}{c}a-c+1, b-c+1\\ 2-c\end{array};x\right)\]
   are orthonormal with respect to the inner product
   \[\langle  f,g\rangle_\omega = \int_0^1 f(x)\bar{g}(x)\omega(x)\, dx\]
   where
    \[\omega(x)=x^{c-1}(1-x)^{a+b-c}\]
    \end{lemma}

With this in mind, we can compute the following.
\begin{theorem}
    If $\tilde{y}(x)=y_1(x)+\mathcal{O}(\rho)$ is an eigenfunction of $\widetilde{L}$ with eigenvalue $\tilde{\lambda}_\rho=ab + \lambda_1\rho+\mathcal{O}(\rho)^2$, then
    \begin{equation}
\lambda_1 = \int^a_b|y_1|^2\omega (x)f(x)dx
\label{3.5}
    \end{equation}
\end{theorem}

\begin{myproof}
From the equation \eqref{3.4}
    \[[\widetilde{L}+\rho f(x)]\Tilde{y}=\Tilde{\lambda}\Tilde{y}\]
and so
   \[[\widetilde{L}+\rho f(x)]\Tilde{y}=(ab+\rho \lambda_1+\mathcal{O}(\rho^2))\Tilde{y}.\] 
Expanding $\tilde{y}$ as a series
\[\tilde{y}=y_{1}+\rho y_{1,1} +\rho^2y_{1,2}+\mathcal{O}(\rho^3)\]
we can compare coefficients of $\rho$ to find
\[\begin{cases}
    \widetilde{L}y_1 = ab y_1\\
    \rho \widetilde{L}y_{1,1}+\rho f(x)y_1 = ab\rho y_{1,1} +\rho\lambda_1 y_1\\
    \vdots
\end{cases}\]
In particular,
\begin{equation}
     \widetilde{L}y_{1,1}+ f(x)y_1 =  ab y_{1,1}+\lambda_1 y_1
     \label{3.6}
\end{equation}
The equation \eqref{3.6} can be rewritten as 
\[\widetilde{L}y_{1,1}-aby_{1,1}=\lambda_1y_1-f(x)y_1\]
We then compute the inner product with ${y_1}$
\begin{equation}
<\widetilde{D}y_{1,1},y_1>_{\omega} = \lambda_1<y_1,y_1>_\omega-<f(x)y_1,y_1>_\omega
\label{3.7}
\end{equation}
where ${\widetilde{D}}=\widetilde{L}-ab$.

We first discuss the left-hand side of the equation \eqref{3.7}. Via integration by parts, one can show
\[<\widetilde{D}y_{1,1},y_1>_{\omega} = <y_{1,1},\widetilde{D}y_1> =0\]
as \(\widetilde{D}y_1=0\) by definition.

The result then follows from Lemma \ref{ortho}
\end{myproof}

From the equation \eqref{3.5}, we get a linear functional $f(x)\mapsto \lambda_1$, given by integration against the density function
\[\rho(x)=|y_1|^2\omega(x)\]

\begin{proposition}[Boundedness under $L^2(\omega)$ normalization]\label{prop:bounded}
Assume that $\langle f,f\rangle_\omega=1$. Then
\[
\lambda_1 \le \Big(\int_D \frac{\rho(x)^2}{\omega(x)}\,dx\Big)^{1/2}.
\]
with equality if and only if $f(x)=C|y_1(x)|^2$ for some constant $C$.
\end{proposition}

\begin{myproof}
By the Cauchy-Schwarz inequality, we know that  
\[\lambda_1 \leq \left(\int_0^1dx\frac{\rho^2}{\omega}\right)^{\frac{1}{2}}\left(\int_0^1dx\omega f^2\right)^{\frac{1}{2}}=\left(\int_0^1dx\frac{\rho^2}{\omega}\right)^{\frac{1}{2}}\]
with equality if and only if $f(x)=C\rho(x)/\omega(x) = C|y_1(x)|^2$.
\end{myproof}

\vspace{1.5\baselineskip}

  \bibliographystyle{plain}
  \bibliography{bibliography}

\appendix
\section{Appendix -Derivation of the general form of the n-th order hypergeometric equation}
In this section, we will derive ${\eqref{eq1.2}}$.

\begin{lemma}
    The operator ${(z\frac{d}{dz})^k}$ is equal to following operator
    \begin{equation}
        \sum^k_{n=0}c(k,n)z^n(\frac{d}{dz})^n
        \label{A1}
    \end{equation}
    \label{lemma A}
    where $c(k,n)$ is a Stirling number of the second kind.
\end{lemma}
\begin{proof}
    We will prove this by induction. For $k=1$
    \begin{equation}
        z\frac{d}{dz} = [c(1,0)+c(1,1)z\frac{d}{dz}]=z\frac{d}{dz}.
    \end{equation}
    If ${\eqref{A1}}$ is true for the order $k$ operator, then for $k+1$ we have 
\begin{equation}
    \left(z\frac{d}{dz}\right)^{k+1} = \left(z\frac{d}{dz}\right)\left(z\frac{d}{dz}\right)^k=\left(z\frac{d}{dz}\right)\sum^k_{n=0}c(k,n)z^n\left(\frac{d}{dz}\right)^n=\sum^k_{n=0}c(k,n)\left(z\frac{d}{dz}\right)[z^n\frac{d^n}{dz^n}]
    \label{A3}
\end{equation}
The right hand side is equal to 
\begin{equation}
    \sum^k_{n=0}c(k,n)z\left(nz^{n-1}\frac{d^n}{dz^n}+z^n\frac{d^{n+1}}{dz^{n+1}}\right)=\sum^k_{n=0}nz^n(\frac{d^n}{dz^n})+\sum^k_{n=0}c(k,n)z^{n+1}\frac{d^{n+1}}{dz^{n+1}}
    \label{A4}
\end{equation}
Collecting like terms, we find
\begin{equation}
        \sum^{k+1}_{m=0}mc(k,m)z^m(\frac{d^m}{dz^m})+\sum^{k+1}_{m=0}c(k,m-1)z^{m}\frac{d^{m}}{dz^{m}}=\sum^{k+1}_{m=0}[mc(k,m)+c(k,m-1)]z^m\frac{d^m}{dz^m}
        \label{A6}
\end{equation}
Thus, we have the recurrence
\begin{equation}
    c(k,m)m+c(k,m-1) = c(k+1,m)
\end{equation}
which is precisely the defining recurrence of the Stirling numbers. The claim follows.
\end{proof}

\begin{theorem}
The left hand side expansion of Equation ${\eqref{eq 1.1}}$ can be expressed as
\begin{equation}
       \prod_{j=1}^{p} \left( x \frac{d}{dx} + a_j \right)y  =\sum^p_{n=0}[\sum^p_{k=n}c(k,n)e_{p-k}(a_1,...,a_p)z^n]\frac{d^n}{dz^n}y
       \label{A8}
\end{equation}
\end{theorem}

\begin{myproof}
   Writing $\theta=x\frac{d}{dx}$, we can expand the product in terms of elementary symmetry polynomials \cite{macdonald1998symmetric}
   \begin{equation}
       \prod_{j=1}^{p} \left( \theta+ a_j \right)= \sum^p_{k=0}e_{p-k}(a_1,..,a_p)\theta^k = \sum^p_{k=0}e_{p-k}(a_1,..,a_p)\sum^m_{n=0}c(k,n)z^n\frac{d^n}{dz^n}
   \end{equation}
   \end{myproof}
   
\begin{lemma}
 For a polynomial ${f(x)}$,
    \begin{equation}
        \frac{df}{dx} = f(\theta+1)\frac{d\theta}{dx}
    \end{equation}
    for the differential operator $\theta=x\frac{d}{dx}$.
    \label{lemma A4}
\end{lemma}
\begin{proof}
Consider ${f(\theta)}$
\begin{equation}
    f(\theta) = \sum^p_na_n\theta^n
    \label{A11}
\end{equation}
Consider the commutator
\begin{equation}
  [\theta,\frac{d}{dz}]= \theta \frac{d}{dz}- \frac{d}{dz}\theta = z\frac{d}{dz}\frac{d}{dz}-\frac{d}{dz}(z\frac{d}{dz})=-\frac{d}{dz}
\end{equation}
We then need the following claim.

\begin{claim}
Given two linear operators $A$, $B$such that [A,B] = ${\lambda}$B then the following identity can be found.
\begin{equation}
    BA^n=(A-\lambda)^nB
\end{equation}
\label{claim A4}
\end{claim}

This claim follows easily by induction, and so 
\begin{equation}
    \frac{df(\theta)}{dz}=\sum^p_na_n\frac{d\theta^n}{dz}= \sum^p_n  a_n(\theta+1)^n\frac{d}{dz}=f(\theta+1)\frac{d}{dz}
\end{equation}
\end{proof}
   \begin{theorem}
    The right hand side of equation${\eqref{eq 1.1}}$ is 
    \begin{equation}
 \sum_{n=1}^{p} \left[
\sum_{k=n-1}^{p-1}
e_{p-k-1}(b_1 - 1,...,b_{p-1}-1)
\sum_{j=n-1}^{k}
\binom{k}{j}
s(j, n-1)
\right]
z^{n-1} \frac{d^n}{dz^n}
    \end{equation}
   \end{theorem} 
\begin{myproof}
    \begin{equation}
        \frac{d}{dz}\prod_{k=1}^{p-1} \left( x \frac{d}{dz} + b_k -1 \right) y(x) = \frac{d}{dz}[\sum^{p-1}_{k}e_{p-k-1}(b_1-1,b_2-1,...,b_{p-1})(z\frac{d\theta}{dz})^k]
    \end{equation}
    \begin{equation}
        \frac{d}{dz}[\sum^{p-1}_{k}e_{p-k-1}(b_1-1,b_2-1,...,b_{p-1})\theta^k = \sum^{p-1}_{k}e_{p-k-1}(b_1-1,b_2-1,...,b_{p-1})\frac{d\theta^k}{dz}
        \label{A.19}
    \end{equation}
    By using lemma${\eqref{lemma A4}}$ this is equal to
    \begin{equation}
        \sum^{p-1}_{k}e_{p-k-1}(b_1-1,b_2-1,...,b_{p-1})(\theta+1)^k\frac{d}{dz}
        \label{A22}
    \end{equation}
    For Equation ${\eqref{A22}}$ the binomial expansion can be used to find ${(\theta+1)^k}$. Swapping the order of summation and simplifying, we then find

\begin{equation}
 \sum_{n=1}^{p} \left[
\sum_{k=n-1}^{p-1}
e_{p-k-1}(b_1 - 1,...,b_{p-1}-1)
\sum_{j=n-1}^{k}
\binom{k}{j}
c(j, n-1)
\right]
z^{n-1} \frac{d^n }{dz^n}
\end{equation}

where $c(j,n-1)$ is the Stirling number of the second kind, and $e_{p-k-1}(b_j - 1)$ is the elementary symmetric polynomial evaluated at $b_j - 1$.
Therefore we obtain 

\begin{equation}
\sum_{n=1}^{p} \left[
\sum_{k=n-1}^{p-1}
e_{p-k-1}(b_1 - 1,...,b_{p-1}-1)
\sum_{j=n-1}^{k}
\binom{k}{j}
c(j, n-1)
\right]
z^{n-1} \frac{d^n }{dz^n}
\end{equation}

\end{myproof}

Combining these, we obtain Equation ${\eqref{eq1.2}}$.

\section{Appendix-Series solution of nth-order deformed differential equation}
We consider a deformation of the $n$\textsuperscript{th} order differential equation by introducing a perturbation in its zeroth-order term. Specifically, we replace the equation:
\begin{equation}
    A_n(x)y^{(n)}+.....+A_0(x)y=0
    \label{homogenous}
\end{equation}
with the deformed version:
\begin{equation}
    A_n(x)y^{(n)}+.....+(A_0(x)+\rho f(x))y=0
    \label{deformed}
\end{equation}
where $\rho$ is a small or formal parameter and $f(x)$ is a continuous or holomorphic function.\\

 In general, the solution of the homogeneous equation ${\eqref{homogenous}}$ would be noted as a set $\{y_1,...,y_n\}$. However, for the deformed equation ${\eqref{deformed}}$, the solution ${y_{\rho}}$ can be expanded into Taylor series form
 \begin{equation}
     y_{\rho}=\sum y_n\rho^n
     \label{expand}
 \end{equation}
Therefore,  by expanding equation ${\eqref{deformed}}$, one can regroup the coefficients in front of \( \rho^n \) and set them to zero. This procedure yields a hierarchy of equations that can be solved recursively, order by order in \( \rho \):

    \begin{equation}
\left\{
\begin{aligned}
& A_n(x)y^{(n)}_0+.....+A_0(x)y_0 = = 0, \\
& A_n(x)y^{(n)}_1+.....+A_0(x)y_1 =-\rho y_0f(x), \\
&\quad \vdots \\
& A_n(x)y^{(n)}_n+.....+A_0(x)y_n =-\rho y_{n-1}f(x),
\end{aligned}
\right.
\label{general}
\end{equation}   
For the leading order of equation ${\eqref{general}}$, it is homogeneous and we asssume we know a basis of $n$ solutions. However, for the higher-order expansion, it is a non-homogeneous differential equation; as a result, we can use the variation of parameters method.
\begin{lemma}
    For a non-homogeneous differential equation, the particular solution could be found by the variation of parameters method.
    \label{lemma 2.2}
\end{lemma}
\begin{myproof}
    We consider the inhomogeneous hypergeometric-type differential equation
    \begin{equation}
        A_n(x)y^{(n)}+.....+A_0(x)y =-\rho yf(x)
        \label{inhyper}
    \end{equation}

This equation satisfies the conditions required for applying the variation of parameters method:
\begin{itemize}
    \item It is a linear differential equation.
    \item The homogeneous counterpart (with $f(x) = 0$) has $n$ known linearly independent solutions, (typically expressed in terms of $_{n}F_{n-1}$ hypergeometric functions in the hypergeometric case).
    \item The Wronskian determinant $\det(W(x))$ is nonzero on intervals excluding the singular points $x = 0$, $x = 1$, and possibly $\infty$.
\end{itemize}

Therefore, the variation of parameters method is valid for constructing a particular solution to the equation.
\end{myproof}
We assume a particular solution ${y_{p}}$ can be expressed as
\begin{equation}
y_p(x) = \sum_{i=1}^n u_i(x)y_i(x)
\label{Varition method}
\end{equation}
where \( u_1(x), \dots, u_n(x) \) are differentiable functions known as \textit{variation parameters}, to be determined.
This is a linear system of \( n \) equations in the \( n \) unknowns \( u_1'(x), \dots, u_n'(x) \), with coefficient matrix given by the Wronskian matrix:
\[
W(x) = \begin{pmatrix}
y_1(x) & y_2(x) & \cdots & y_n(x) \\
y_1'(x) & y_2'(x) & \cdots & y_n'(x) \\
\vdots & \vdots & \ddots & \vdots \\
y_1^{(n-1)}(x) & y_2^{(n-1)}(x) & \cdots & y_n^{(n-1)}(x)
\end{pmatrix}.
\]

To ensure that the expression for \( y(x) \) does not contribute extra derivatives beyond order \( n \), we impose the following system of constraints:
\[
\begin{cases}
\sum_{i=1}^n u_i'(x)y_i(x) = 0, \\
\sum_{i=1}^n u_i'(x)y_i'(x) = 0, \\
\quad\vdots \\
\sum_{i=1}^n u_i'(x)y_i^{(n-2)}(x) = 0, \\
\sum_{i=1}^n u_i'(x)y_i^{(n-1)}(x) = f(x).
\end{cases}
\]
or equivalently
\begin{equation}
    \begin{pmatrix} 
        y'_1 & ...& y'_n\\
        y''_1 & ...&y''_n \\
         &   \cdots\\
          & \cdots \\
       y^{(n-1)}_1 & \cdots & y^{(n-1)}_n
    \end{pmatrix} \begin{pmatrix}
        u'_1\\
        .\\
        .\\
        .\\
        u'_n
    \end{pmatrix} = \begin{pmatrix}
        0\\
        0\\
        0\\
        .\\
        f(x)
    \end{pmatrix}
\end{equation}
Then, by using Cramer's rule, ${u'_i}$ can be easily computed 

\begin{equation}
u_i'(x) = \frac{
\begin{vmatrix}
y_1(x) & \cdots & y_{i-1}(x) & 0 & y_{i+1}(x) & \cdots & y_n(x) \\
y_1'(x) & \cdots & y_{i-1}'(x) & 0 & y_{i+1}'(x) & \cdots & y_n'(x) \\
\vdots & & \vdots & \vdots & \vdots & & \vdots \\
y_1^{(n-2)}(x) & \cdots & y_{i-1}^{(n-2)}(x) & 0 & y_{i+1}^{(n-2)}(x) & \cdots & y_n^{(n-2)}(x) \\
y_1^{(n-1)}(x) & \cdots & y_{i-1}^{(n-1)}(x) & f(x) & y_{i+1}^{(n-1)}(x) & \cdots & y_n^{(n-1)}(x)
\end{vmatrix}
}{
det(W(y_1, y_2, \ldots, y_n)(x))
}
\end{equation}
Therefore, ${u_i}$ can be found by using integration
\begin{equation}
    u_i = \int \frac{det\Omega}{detW} dx \ \ \ \ \ \ \ \ \Omega= \begin{vmatrix}
y_1(x) & \cdots & y_{i-1}(x) & 0 & y_{i+1}(x) & \cdots & y_n(x) \\
y_1'(x) & \cdots & y_{i-1}'(x) & 0 & y_{i+1}'(x) & \cdots & y_n'(x) \\
\vdots & & \vdots & \vdots & \vdots & & \vdots \\
y_1^{(n-2)}(x) & \cdots & y_{i-1}^{(n-2)}(x) & 0 & y_{i+1}^{(n-2)}(x) & \cdots & y_n^{(n-2)}(x) \\
y_1^{(n-1)}(x) & \cdots & y_{i-1}^{(n-1)}(x) & f(x) & y_{i+1}^{(n-1)}(x) & \cdots & y_n^{(n-1)}(x)
\end{vmatrix}
\label{equation 1.8}
\end{equation}
The particular solution can then be calculated using the equation \eqref{equation 1.8}. 
\begin{theorem}
    The full solution of the deformed nth order hypergeometric equation is 
    \begin{equation}
        y_{deformed} = \sum^k_{n=0}y_{h,n}\rho^n + \sum^k_{n=1}y_{p,n}\rho^n
        \label{equation 2.1}
    \end{equation}
    \label{theorem 2.3}
      where ${y_{h,n}}$ is a general solution to the homogeneous differential equation, $y_{p,n}$ is a solution to the corresponding inhomogeneous differential equation
\end{theorem}

\subsection{Example-second order case}
In this example, we consider a function \( f(x) \). When a small perturbative term \( \rho f(x) \), with \( \rho \) being a small parameter, is added to the coefficient of the zeroth-order derivative
\begin{equation}
   x(1-x)y''+[c-(a+b+1)x]y'-(ab+\rho f(x))y = 0 
   \label{2.1}
\end{equation}
 The solution of the entire differential equation changes. The main goal of this section is to study the correction to the solution induced by this modification.
From the equation${\eqref{expand}}$ can be expressed as:\\
\begin{equation}
   x(1-x)\sum^{\infty}_{n=0} y''_n\rho^n+[c-(a+b+1)x]\sum^{\infty}_{n=0} y'_n\rho^n-(ab+\rho f(x))\sum^{\infty}_{n=0} y_n\rho^n = 0  
   \label{2.3}
\end{equation}
In general, by expanding equation ${\eqref{2.3}}$, one can regroup the coefficients in front of \( \rho^n \) and set them to zero. This procedure yields a hierarchy of equations that can be solved recursively, order by order in \( \rho \).
\begin{equation}
\left\{
\begin{aligned}
& y_0'' +\frac{[c - (a + b + 1)x]}{x(1-x)} y_0' - \frac{ab}{x(1-x)}\, y_0 = 0, \\
& y_1'' + \frac{[c - (a + b + 1)x]}{x(1-x)} y_1' -\frac{ab}{x(1-x)}\, y_1 = \frac{y_0}{x(1-x)}, \\
&\quad \vdots \\
&y_k'' + \frac{[c - (a + b + 1)x]}{x(1-x)} y_k' - \frac{ab}{x(1-x)}\, y_k = \frac{y_{k-1}}{x(1-x)}
\end{aligned}
\right.
\label{2.4}
\end{equation}
The general solution of the equation ${\eqref{2.4}}$ can be found in ${\eqref{1.7}}$, then we can get the following
set of equations
\begin{equation}
\begin{cases}
y_{h0}(x) = A_{01} \,{}_2F_1(a, b; c; x) + A_{02} x^{1 - c} \,{}_2F_1(a - c + 1, b - c + 1; 2 - c; x) \\
y_{h1}(x) = A_{11} \,{}_2F_1(a, b; c; x) + A_{12} x^{1 - c} \,{}_2F_1(a - c + 1, b - c + 1; 2 - c; x) \\
\vdots \\
y_{hk}(x) = A_{k1} \,{}_2F_1(a, b; c; x) + A_{k2} x^{1 - c} \,{}_2F_1(a - c + 1, b - c + 1; 2 - c; x)
\end{cases}
\end{equation}
Therefore, following the theorem ${\eqref{theorem 2.3}}$, it is sufficient to compute the homogeneous solution in series form.
\begin{equation}
    y_h = \sum^{\infty}_{n=0}y_{nh} \rho^n
\end{equation}

\begin{theorem}
    Suppose there exists ${u_1}$ and ${u_2}$, such that ${y_{\rho i}}$ a particular solution of equation ${\eqref{2.4}}$ can be expressed as 
    \begin{equation}
        y_{\rho i} = u_1y_1+u_2y_2
        \label{equation 2.8}
    \end{equation}
    Imposing the constraint
    \begin{equation}
        u'_1y_1 +u'_2y_2= 0
    \end{equation}
    we can determine ${u_1}$ as
    \begin{equation}
      u_1= - \int\frac{f(x)}{A}x^{c}(1-x)^{c-(a+b+1)}y_2dx
        \label{Wroskian}
    \end{equation}
    \label{theorem 2.4}
    for some constant $A$, and similarly for $u_2$.
\end{theorem}

\begin{myproof}
From Lemma ${\eqref{lemma 2.2}}$, the non-homogeneous hypergeometric equation satisfies the condition of the variation parameter method; therefore, the particular solution can be expanded into the equation ${\eqref{equation 2.8}}$ form with constraint
\begin{equation}
u'_1(x)y_1(x)+u'_2(x)y_2(x)=0    
\end{equation}
and 
\begin{equation}
    y_0(x) = u'_1(x)y'_1(x)+u'_2(x)y'_2(x)
\end{equation} 
We can apply Abel's theorem to find $\det(W)$:
\begin{equation}
   \det(W) = Cexp(-\int\frac{[c - (a + b + 1)x]}{x(1-x)} dx)
   \label{exp 2.13}
\end{equation}
for some constant $C$. Hence
\begin{equation}
    det(W) = \frac{A}{x^{c}(1-x)^{(a+b+1)-c}}
\end{equation}
for some constant $A$. We can similarly compute $\det(\Omega)$, to obtain Equation ${\eqref{Wroskian}}$.
\end{myproof}

\begin{theorem}
   The ith-order particular solution ${y_{pi}}$ of equation ${\eqref{2.4}}$  is  
   \begin{equation}
- \frac{y_{i1}}{A}\int y_{i2}y_{i-1}  \, x^{c}(1 - x)^{c - (a + b + 1)} \, dx
+ \frac{y_{i2}}{A} \int y_{i1} y_{i-1}  \left( x^{c}(1 - x)^{c - (a + b + 1)} \right) dx = y_{pi}
\label{2.15}
\end{equation}
\end{theorem}
\begin{myproof}
    By using theorem ${\eqref{theorem 2.3}}$, every terms of the equation ${\eqref{equation 2.1}}$ can be easily expressed into equation${\eqref{2.15}}$ form.
\end{myproof}
Therefore, by using equation ${\eqref{equation 2.1}}$ , we can find the full solution of deformed equation${\eqref{2.1}}$.\\
The homogenous part of equation ${\eqref{equation 2.1}}$ is a power series of a hypergeometric function, and for a particular solution, it is a series of integration of f(x) times the hypergeometric equation; some integration techniques of hypergeometric functions can be found in \cite{riley2006mathematical}.


\Needspace*{3\baselineskip}
\noindent
\rule{\textwidth}{0.15em}
ZIYU ZHANG\\
{\noindent\small
Trinity College Dublin, School of Mathematics
\\
17 Westland Row
Trinity College Dublin
Dublin 2
Ireland\\
E-mail: \url{zhangz6@tcd.ie}
}\vspace{.5\baselineskip}

\end{document}